\documentclass[leqno]{article}
\usepackage{geometry}
\usepackage{graphicx}	
\usepackage[cp1251]{inputenc}

\usepackage[english]{babel}
\usepackage{mathtools}
\usepackage{amsfonts,amssymb,mathrsfs,amscd,amsmath,amsthm}

\usepackage{verbatim}

\usepackage{url}

\usepackage{stmaryrd}
\usepackage{chebsb}
\ifpdf
\usepackage{hyperref}
\else
\usepackage[dvipdfmx]{hyperref}
\fi

\begin{document}
\title{On the Gromov--Hausdorff distance between the cloud of bounded
metric spaces and a cloud with nontrivial stabilizer}
\author{B.\,A.~Nesterov}
\date{}
\maketitle

\begin{abstract}
The paper studies the class of all metric spaces considered up to zero Gromov--Hausdorff distance between them. In this class, we examine clouds --- classes of spaces situated at finite Gromov--Hausdorff distances from a reference space. We prove that all clouds are proper classes. The Gromov--Hausdorff distance is defined for clouds similarly with the case of that for metric spaces. A multiplicative group of transformations of clouds is defined which is called stabilizer. 
    We show that under certain restrictions the distance between the cloud of bounded metric spaces and a cloud with a nontrivial stabilizer is finite. In particular, the distance between the cloud of bounded metric spaces and the cloud containing the real line is calculated.
    
\end{abstract}

\section{Intodiction}
\markright{\thesection.~Introduction}

The present paper is devoted to the study of the Gromov--Hausdorff distance ~\cite{Edwards, Gromov81, Gromov99},
defined on the proper class of all non-empty metric spaces considered up to isometry. It is known that in this class, the distance is a generalized pseudometric i.e., the distance may be zero for non-equal elements, and the distance can take infinite values.  

Traditionally, the Gromov--Hausdorff distance is studied on the class of compact metric spaces considered up to isometry. This class is called the Gromov--Hausdorff space. On it, the distance becomes a metric. Below, the Gromov--Hausdorff distance between spaces $ X $ and $ Y $ will be denoted by $ d_{GH}(X, Y) $ or $ |X, Y| $.

M. Gromov introduced this distance in \cite{Gromov81} and used it to prove the theorem on groups of polynomial growth.  

Later, this distance found applications in computational geometry, where it was used for shape matching and similarity measurement \cite{memoli1}. The Gromov--Hausdorff distance can also be applied in robotics for motion planning \cite{robotics}.  

Computing the Gromov--Hausdorff distance is algorithmically an NP-hard problem, and to simplify calculations, the distance is often modified, see, for example, \cite{memoli2}.  

In \cite{Gromov99}, M. Gromov also considered the Gromov--Hausdorff distance on classes of unbounded spaces that are at a finite distance from each other. In the present paper these classes will be called \emph{clouds} and are the primary subject of the present work. Gromov stated that all clouds are complete and contractible but he did not provide proofs for these facts \cite{Gromov99}. Subsequently, S. A. Bogatyi and A. A. Tuzhilin in \cite{TuzhBog1} proved the completeness of clouds. However, the problem of contractibility turned out to be significantly more challenging.  

First, we note that contractibility is a topological concept, as it relies on continuous mappings. Recall that in the von Neumann-Bernays-G\"odel (NBG) set theory, every object is a class, which can either be a set or a proper class. By definition, a set is an element of another class \cite{Neumann, Bernays, Godel}. A proper class cannot ba an element of another class. Thus, for proper classes, it is impossible to define a topology in the usual sense, since the class itself would then have to be an element of it. In \cite{BorIvTuzh1}, a generalized notion of topology and continuous mappings for proper classes was introduced using the concept of filtration by sets. If such a filtration exists in a class, the class is called \emph{topological}. In the present work, we prove that every cloud is a proper class. Therefore, to meaningfully discuss the contractibility of clouds, a generalization of topology is required, as, for example, developed in \cite{BorIvTuzh1}.

However, generalizing topology alone is not sufficient. To illustrate this, we introduce several additional concepts.  
For any metric space, one can define an operation of multiplication by a positive real number $\lambda$. Under this operation $H_{\lambda}\colon X \to \lambda X$, all distances in the metric space $X$ are scaled by $\lambda$. Moreover, in the case of bounded metric spaces, this operation can be extended to zero by setting $0 \cdot X := \Delta_1$, where $\Delta_1$ is a one-point metric space.  

It is well known that for any bounded spaces $X$, $Y$ and non-negative real numbers $\lambda$, $\mu$, the following holds:  
$$
|\lambda X, \mu X| = |\lambda - \mu| \cdot |X, \Delta_1| = \frac{1}{2} |\lambda - \mu| \cdot \text{diam}\, X,
$$  
where $\text{diam}\, X$ is the diameter of $X$. Also, for any metric spaces $X,Y$ the following holds:
$$
|\lambda X, \lambda Y| = \lambda |X, Y|,
$$  
  
From these properties, it is straightforward to show that the cloud of bounded metric spaces is indeed contractible.  

However, if we consider a cloud containing $\mathbb{R}^n$, the operation $H_\lambda$ maps the cloud into itself for all $\lambda$ but it is discontinuous at certain points. Moreover, there exist spaces that, when multiplied by some positive real numbers, are mapped to spaces at infinite Gromov--Hausdorff distance \cite{TuzhBog1}. This implies that the clouds containing them are not preserved under such scaling.  
From the properties mentioned above, it follows that if multiplying a space by $\lambda$ keeps it within its own cloud, then all spaces in that cloud also remain within it. Moreover, if a space transitions into another cloud under such scaling, then all spaces from the source cloud transition to the same target. Thus, the operation of multiplication by $\lambda$ can be naturally extended to clouds themselves.  

As discussed earlier, the mapping $H_\lambda$ posesses nontrivial properties, which motivates further investigation. To study the operation $H_\lambda$, the concept of a \emph{stabilizer} of a cloud was introduced: it is the multiplicative group of all positive $\lambda$ for which $H_\lambda$ maps the cloud into itself. In \cite{TuzhBog2}, the notion of a \emph{center} of a cloud was defined as a space that, under the action of transformations from the stabilizer, maps to a space at zero Gromov--Hausdorff distance from itself. It was also shown that every cloud has a unique center up to zero distances. The concepts of the stabilizer and the center of a cloud play a key role in this work.  

The present study primarily focuses on investigating the Gromov--Hausdorff distance between clouds, one of which is the cloud of bounded metric spaces. We formulate and prove a theorem concerning the image of $\Delta_1$ under a correspondence with finite distortion between the cloud of bounded metric spaces and a cloud with a nontrivial stabilizer. As a corollary, we establish a theorem stating that the distance from the cloud of bounded metric spaces to clouds with nontrivial stabilizers with some additional assumptions is infinite. As an example, we show that the Gromov--Hausdorff distance between the cloud of bounded metric spaces and the cloud which contains the real line is infinite. 

The author expresses deep gratitude to their advisor, A. A. Tuzhilin, and Professor A. O. Ivanov for formulating the problem and for fruitful discussions of the results. 

\section{Preliminaries}
\markright{\thesection.~Preliminaries}

Let $ X $ and $ Y $ be metric spaces. A distance between them, called the \emph{Gromov--Hausdorff distance}, can be defined. Below, we present two equivalent definitions \cite{Lectures}.

\begin{defin}  
Let $ X $ and $ Y $ be metric spaces. A \emph{correspondence} $ R $ between these spaces is a surjective multivalued mapping from $ X $ to $ Y $. The set of all correspondences between $ X $ and $ Y $ is denoted by $ \mathcal{R}(X,Y) $. We will also identify a correspondence with its graph.  
\end{defin}

\begin{defin}  
Let $ R $ be a correspondence between $ X $ and $ Y $. The \emph{distortion} of $ R $ is defined as  
$$
\dis{R} = \sup \Bigl\{ \big| |xx'| - |yy'| \big| : (x, y), (x', y') \in R \Bigr\}.  
$$  
Then, the \emph{Gromov--Hausdorff distance} $ d_{GH}(X,Y) $ can be defined as  
$$
d_{GH}(X,Y) = \frac{1}{2} \inf \bigl\{ \dis{R} : R \in \mathcal{R}(X,Y) \bigr\}.  
$$  
\label{defSootvet}  
\end{defin}
\begin{defin}
  Let $X,Y$ be subsets of a metric space $Z$. We define the Hausdorff distance between $X,Y$ as follows:
  $$
  d_H(X, Y) \coloneqq \max \left\{ \sup_{x \in X} d(x, Y), \sup_{y \in Y} d(X, y) \right\}.
  $$
\end{defin}
\begin{defin}  
A \emph{realization} of a pair of metric spaces $ (X,Y) $ is a triple $ (X', Y', Z) $ of metric spaces such that:  
$ X' \subset Z $ and $ Y' \subset Z $,  
$ X' $ is isometric to $ X $ and 
$ Y' $ is isometric to $ Y $.  

The \emph{Gromov--Hausdorff distance} $ d_{GH}(X,Y) $ between $ X $ and $ Y $ is the infimum of all numbers $ r $ for which there exists a realization $ (X', Y', Z) $ satisfying $ d_H(X', Y') \leq r $, where $ d_H $ is the Hausdorff distance.  
\end{defin}  

Henceforth, the Gromov--Hausdorff distance between metric spaces $ X $ and $ Y $ will be denoted by $ |X,Y| $.  

 We will consider two metric spaces equivalent if they are a zero Gromov--Hausdorff distance from each other. The resulting class is denoted by $ \mathcal{GH}_0 $. On this class, the Gromov--Hausdorff distance becomes a \emph{generalized metric}.  

\begin{defin}[\cite{TuzhBog1}]
In the class $\mathcal{GH}_{0}$, we consider the following relation: $X \thicksim Y \Leftrightarrow d_{GH}(X, Y) < \infty$. It is easy to verify that this is an equivalence relation. The equivalence classes of this relation are called \emph{clouds}. The cloud containing a metric space $X$ will be denoted by $[X]$.
\end{defin}

For any metric space $X$, we can define an operation of multiplication by a positive real number $\lambda\colon X\mapsto \lambda X$, specifically $(X, \rho) \mapsto (X, \lambda \rho)$, where the distance between any two points of the space is scaled by a factor of $\lambda$.

\begin{remark}
Let metric spaces $X$, $Y$ belong to the same cloud. Then $d_{GH}(\lambda X, \lambda Y) = \lambda d_{GH}(X,Y)\allowbreak < \infty$, which means the spaces $\lambda X$, $\lambda Y$ will also belong to the same cloud.
\label{remOneCloud}
\end{remark}

\begin{defin}
We define the operation of multiplying a cloud $[X]$ by a positive real number $\lambda$ as the mapping that transforms all spaces $Y \in [X]$ into spaces $\lambda Y$. According to Remark \ref{remOneCloud}, all resulting spaces will belong to the cloud $[\lambda X]$.
\end{defin}

Under such a mapping, a cloud may either change or remain invariant. For the latter case, we introduce a special definition.

\begin{defin}[\cite{TuzhBog2}]
The \emph{stabilizer} $\St\bigl([X]\bigr)$ of a cloud $[X]$ is the subset of $\mathbb{R}_+$ such that for all $\lambda \in \St\bigl([X]\bigr)$, $[X] = [\lambda X]$. This subset forms a subgroup of $\mathbb{R}_+$. The stabilizer is called \emph{trivial} when it equals $\{1\}$.
\end{defin}

Let us provide several examples of clouds and their stabilizers.

\begin{itemize}
    \item Let $\Delta_1$ be a one-point metric space. Then $\St\bigl([\Delta_1]\bigr) = \mathbb{R}_+$.
    \item $\St\bigl([\mathbb{R}]\bigr) = \mathbb{R}_+$.
    \item Suppose a function $\phi(n)$ satisfies $\lim\limits_{n \rightarrow \infty} \phi(n + 1) - \phi(n) = +\infty$. For $q > 1$, define the space $X_q = \left\{q^{\phi(n)}: n \in \mathbb{N}\right\}$. Then $\St\bigl([X_q]\bigr) = \{1\}$ \cite{TuzhBog1}.
    \item For a natural number $p$, define the space $X_p = \left\{p^n : n \in \mathbb{Z}\right\}$. For any prime $p$, we have $\St\bigl([X_p]\bigr) = \left\{p^n : n \in \mathbb{Z}\right\}$ \cite{BogBog1}.
\end{itemize}

\begin{lemma}[\cite{TuzhBog2}]
    In every cloud with a nontrivial stabilizer, there exists a unique space $X$ such that for any $\lambda$ from the stabilizer, $X = \lambda X$ holds.
    \label{centerLemma}
\end{lemma}

\begin{defin}
	The metric space from the Lemma \ref{centerLemma} will be called the \emph{center} of the cloud.
\end{defin}
\begin{remark} For any metric space $X$ in the cloud $[\Delta_1]$:
	$$|\lambda X, \mu X| = |\lambda - \mu|\cdot|X,\Delta_1|.$$
\end{remark} \begin{remark}[Ultrametric inequaliy] For any metric spaces $X_{1}, X_{2}$ from the cloud$[\Delta_{1}]$ the following inequality holds:
	$$|X_{1},X_{2}| \le \max\{|X_{1}, \Delta_{1}|,|X_{2},\Delta_{1}|\}.$$
	\label{remUltraMetric}
\end{remark}
\section{Cloud cardinality}
\markright{\thesection.~Cloud cardinality}

By definition, metric spaces are sets. Therefore, to extend the Gromov--Hausdorff distance construction to clouds, we must either establish that clouds are sets or appropriately modify the distance definition.

We employ the lemma concerning the nature of cardinal number sets:

\begin{lemma}[\cite{levySet}]
Any set of cardinal numbers has an upper bound.
\end{lemma}

The following corollary from this lemma will be necessary for our proof:

\begin{corollary}\label{colCardinal}
The class of unbounded cardinals is proper.
\end{corollary}

We now formulate and prove the theorem about the class of spaces within each cloud:

\begin{theorem}
All clouds are proper classes.
\end{theorem}

\begin{proof}
Following Corollary \ref{colCardinal}, it suffices to show that any cloud contains spaces of arbitrarily large cardinality. Let $X$ be a metric space of cardinality $\alpha$. We extend this space to one of greater cardinality $\beta > \alpha$ by constructing $X_\beta = X \cup \Delta_\beta$, where $\Delta_\beta$ is a simplex of cardinality $\beta$.

Fix an arbitrary point $x \in X$ and set the distance from $x$ to any simplex point as 1. For $x' \in X$ and $y \in \Delta_\beta$, define:
$$
\rho_{X_\beta}(y,x') = \rho_{X_\beta}(x',y) := \rho_X(x',x) + 1.
$$
All other distances remain unchanged.

The resulting space $X_\beta$ is indeed metric:
\begin{enumerate}
\item Symmetry and non-negativity are obvious
\item The triangle inequality holds for all cases:
\begin{itemize}
\item When $x', z' \in \Delta_\beta$ or $x', z' \in X$: trivial.
\item When $x' \in X$, $z' \in \Delta_\beta$:
\begin{itemize}
\item if $y' \in X$: $\rho_{X_\beta}(x',z') = \rho_X(x,x') + 1 \leq \rho_X(x,y') + \rho_X(y',x') + 1$,
\item if $y' \in \Delta_\beta$: $\rho_{X_\beta}(x',z') = \rho_X(x,x') + 1 \leq \rho_X(x',x) + 2$.
\end{itemize}
\end{itemize}
\end{enumerate}

Since $X$ can be isometrically embedded into $X_\beta$ with $X_\beta$ lies in a closed 1-neighborhood of $X$, their Gromov--Hausdorff distance is finite.
\end{proof}

\begin{remark}
As all clouds are proper classes, a bijection exists between any two clouds. This implies, in particular, that the class of correspondences between any two clouds is non-empty.
\end{remark}

\begin{defin}
Let $\mathcal{R}([X],[Y])$ denote the class of all correspondences between clouds $[X]$ and $[Y]$. We define the distortion $\text{dis}\, R$ as in Definition \ref{defSootvet}. The Gromov--Hausdorff distance between clouds is:
$$
d_{GH}([X],[Y]) = \frac{1}{2}\inf\{\text{dis}\, R : R \in \mathcal{R}([X],[Y])\}.
$$
\end{defin}

\section{Center image theorem}
\markright{\thesection.~Center image theorem}

Before proving the main theorem of this section we must list several key properties of correspondences between clouds on which our proof will rely.
\begin{lemma}\label{lemDiamImage}
  If $R$ is a correspondence between two clouds with distortion $r$ then for any spaces $Y_1,Y_2$ which lie in $R(X)$, $|Y_1,Y_2| \le r$.
\end{lemma}

\begin{proof}
If spaces $Y_{1}, Y_{2}$ lie in the image of $X$, then
$$
\operatorname{dis} R \geq \big|\left|Y_{1},Y_{2}\right| - \left|X,X\right|\big| = |Y_{1}, Y_{2}|,
$$
from which $\operatorname{diam} R(X) \leq \operatorname{dis} R$ follows.
\end{proof}

\begin{corollary}
  Suppose than $R$ is a correspondence between two clouds with distortion $r$ and $Y_1,Y_2$ are metric spaces such that $|Y_1,Y_2| > r$. Then $Y_1$ and $Y_2$ cannot lie in the image of a single metric space.
\end{corollary}
Now we are ready to formulate and prove the theorem about the image of $\Delta_1$.
\begin{theorem}\label{thrmCenterImage}
Let $M$ be the center of the cloud $[M]$ with a nontrivial stabilizer. Let $R$ be a correspondence between $[\Delta_{1}]$ and $[M]$ with finite distortion $\varepsilon$. Then any space from the image $R(\Delta_{1})$ lies at a distance from $M$ not exceeding $2\varepsilon$.
\end{theorem}

\begin{proof}
The nontriviality of the stabilizer $[M]$ implies that there exists a number $l > 1$ such that $\{l^{j} \mid j\in \mathbb{Z}\}$ is a subgroup of $\operatorname{St}([M])$.

Fix $Y$ in the image of $\Delta_{1}$. Suppose that $|M, Y| = d > \varepsilon$. Denote $|Y, kY| = \rho$, where $k \geq 2$, $k = l^{j_{1}}$. By the triangle inequality $\rho + d \geq kd$, hence $\rho \geq (k-1)d > (k-1)\varepsilon$. Then $kY$ lies in the image of $X \neq \Delta_1$, with $\rho - \varepsilon \leq |X, \Delta_1| \leq \rho + \varepsilon$.

Take arbitrary $\alpha > 0$ and $\beta \in (0,1)$. For spaces $(1+\alpha)X$, $(1-\beta)X$ the following inequalities hold:
\begin{align*}
|X, (1+\alpha)X| &= \alpha |X, \Delta_1| \leq \alpha\rho + \alpha\varepsilon, \\
|X, (1-\beta)X| &= \beta|X, \Delta_1| \leq \beta\rho + \beta\varepsilon, \\
|(1+\alpha) X, (1-\beta)X| &= (\alpha + \beta)|X, \Delta_1| \geq (\alpha+\beta)\rho - (\alpha+\beta)\varepsilon.
\end{align*}

There exist $Y_\alpha, Y_\beta \in [M]$ such that $kY_\alpha \in R\big((1+\alpha)X\big)$, $kY_\beta \in R\big((1-\beta)X\big)$, and for them the following inequalities hold:
\begin{align*}
|kY, kY_\alpha| &\leq |X, (1+\alpha)X| + \varepsilon \leq \alpha\rho + (\alpha+1)\varepsilon, \\
|kY, kY_\beta| &\leq |X, (1-\beta)X| + \varepsilon \leq \beta\rho + (\beta+1)\varepsilon, \\
|kY_\alpha, kY_\beta| &\geq |(1+\alpha)X, (1-\beta)X| - \varepsilon \geq (\alpha+\beta)\rho - (\alpha+\beta+1)\varepsilon.
\end{align*}

Dividing these inequalities by $k$, we obtain:
\begin{align*}
|Y, Y_{\alpha}| &\leq \frac{\alpha}{k}\rho + \frac{\alpha+1}{k}\varepsilon, \\
|Y, Y_{\beta}| &\leq \frac{\beta}{k}\rho + \frac{\beta+1}{k}\varepsilon, \\
|Y_\alpha, Y_{\beta}| &\geq \frac{\alpha+\beta}{k}\rho - \frac{\alpha+\beta+1}{k}\varepsilon.
\end{align*}

Taking preimages of spaces $Y, Y_{\alpha}, Y_{\beta}$, we obtain:
\begin{align*}
|\Delta_1, X_{\alpha}| &\leq \frac{\alpha}{k}\rho + \left(\frac{\alpha+1}{k} + 1\right)\varepsilon, \\
|\Delta, X_{\beta}| &\leq \frac{\beta}{k}\rho + \left(\frac{\beta+1}{k}+1\right)\varepsilon, \\
|X_\alpha, X_{\beta}| &\geq \frac{\alpha+\beta}{k}\rho - \left(\frac{\alpha+\beta+1}{k}+1\right)\varepsilon.
\end{align*}

Assuming $\alpha > \beta$ we obtain:
$$
\frac{\alpha+\beta}{k}\rho - \left(\frac{\alpha+\beta+1}{k}+1\right)\varepsilon \leq \frac{\alpha}{k}\rho + \left(\frac{\alpha+1}{k} + 1\right)\varepsilon,
$$
which is equivalent to:
$$
\rho \leq \frac{k}{\beta}\left(\frac{2\alpha+\beta+2}{k}+2\right)\varepsilon,
$$
and further:
$$
\rho \leq \left(1+\frac{2\alpha + 2}{\beta} + 2\frac{k}{\beta}\right)\varepsilon.
$$

We are interested in an upper bound for $d$:
$$
d \leq \frac{\rho}{k-1} \leq \left(\frac{1}{k-1}+\frac{2\alpha + 2}{\beta(k-1)} + 2\frac{k}{\beta(k-1)}\right)\varepsilon.
$$

The last term in parentheses is strictly greater than 2 for any $k>2$, $\alpha>0$, $\beta\in (0,1)$, while the other terms tend to 0 as $k$ increases. Since the stabilizer is nontrivial, it contains sequences of numbers tending to both 0 and $\infty$. Taking $\beta$ to 1 and $k$ to infinity, we obtain the estimate:
$$
|Y, M| \leq 2\varepsilon,
$$
which completes the proof.
\end{proof}

\section{Cloud of the real line and the non-utlrametric unequality}
\markright{\thesection.~Cloud of the real line and the non-utlrametric unequality}

\begin{lemma}\label{lemmaDiamDist}
  If $X$ is a subset of the real line and $\mathbb{R}\setminus X$ contains an interval of diameter $2d$, then $|\mathbb{R}, X| \ge d$.
\end{lemma}

\begin{proof}
Suppose that
$\mathbb{R} \setminus X$ contains an interval $(a-d, a+d)$. Suppose $|\mathbb{R}, X|<d$. Let $(\mathbb{R}^{\prime}, X^{\prime}, Y)$ be a realization of $(\mathbb{R}, X)$ with $d_H(\mathbb{R}^{\prime}, X^{\prime}) = d^{\prime}<d$. Define
$$
U_1 := \bigcup_{\substack{x \in X^{\prime} \\ x \leq a-d}} B\left(x, d^{\prime}+\frac{d-d^{\prime}}{2}\right), \quad 
U_2 := \bigcup_{\substack{x \in X^{\prime} \\ x \geq a+d}} B\left(x, d^{\prime}+\frac{d-d^{\prime}}{2}\right),
$$
so $U_1 \cup U_2$ covers $X^{\prime}$ with balls of radius $d^\prime + \frac{d-d^{\prime}}{2}$. 

Then $U_1$ and $U_2$ are two disjoint open sets, but $\mathbb{R}^{\prime} \subset U_1 \cup U_2$, which contradicts the connectedness of the real line.
\end{proof}

For the cloud $[\Delta_{1}]$, Remark \ref{remUltraMetric} shows that the ultrametric inequality holds. The following lemma demonstrates that this inequality may fail for the cloud $[\mathbb{R}]$.

Consider $\mathbb{R}$ as a subset of $\mathbb{R}^2$ and add the point $(0,1)$ with distances given by the $L_1$ metric in $\mathbb{R}^2$. Denote this space by $\widetilde{\mathbb{R}}$.

\begin{theorem}\label{thrmRUltraMetric}
For the spaces $\mathbb{Z}$ and $\widetilde{\mathbb{R}}$, the following hold\/\rom{:}
\begin{enumerate}
  \item $|\mathbb{Z}, \mathbb{R}| \le \frac{1}{2}$, $|\widetilde{\mathbb{R}}, \mathbb{R}| \le \frac{1}{2}$.\label{thrmPt:1}
  \item $|\mathbb{Z}, \widetilde{\mathbb{R}}|>  \frac{1}{2}$.\label{thrmPt:2}
\end{enumerate}
\end{theorem}

\begin{proof}
(1) Embedding $\mathbb{Z}$ in $\mathbb{R}$ gives a realization with Hausdorff distance $\frac{1}{2}$. For $\widetilde{\mathbb{R}}$, embed it naturally in $\mathbb{R}^2$ and $\mathbb{R}$ as $\{(x, \frac{1}{2}) : x \in \mathbb{R}\}$; the Hausdorff distance is again $\frac{1}{2}$.

(2) Let $R$ be a correspondence between $\mathbb{Z}$ and $\widetilde{\mathbb{R}}$ with distortion $1 + \varepsilon$, where $(0,1)$ is in the image of some $i \in \mathbb{Z}$. By Lemma \ref{lemDiamImage}, the image of $i$ must lie in $(-\varepsilon, \varepsilon) \times \{0\} \cup \{(0,1)\}$. 

Let $\mathcal{N}$ be the set of integers whose images lie in $(-\varepsilon, \varepsilon) \times \{0\} \cup \{(0,1)\}$. Then $\mathbb{Z} \setminus \mathcal{N}$ is at \mbox{distance $\geq 1$} from $\mathbb{R}$ by Lemma \ref{lemmaDiamDist}.

Define $R'$ by removing from $R$:
\begin{itemize}
    \item The pair $(i,(0,1))$,
    \item all pairs $(k,x)$ with $x \in (-\varepsilon, \varepsilon) \times \{0\}$.
\end{itemize}

Then $R'$ is a correspondence between $\mathbb{R}\setminus(-\varepsilon, \varepsilon) \times \{0\}$ and $\mathbb{Z} \setminus \mathcal{N}$ with distortion:
$$
1+\varepsilon \geq \operatorname{dis} R' \geq 2\big|\mathbb{R}\setminus(-\varepsilon, \varepsilon), \mathbb{Z} \setminus \mathcal{N}\big|.
$$
By the triangle inequality:
$$
2\big|\mathbb{R}\setminus(-\varepsilon, \varepsilon), \mathbb{Z} \setminus \mathcal{N}\big| \geq 2(1 - \varepsilon).
$$
This yields $1+\varepsilon \geq 2-2\varepsilon$, hence $\varepsilon \geq \frac{1}{3}$ and consequently:
$$
|\widetilde{\mathbb{R}}, \mathbb{Z}| \geq \frac{2}{3} > \frac{1}{2}.
$$
\end{proof}

\section{Gromov--Hausdorff distance betweem clouds}
\markright{\thesection.~Gromov--Hausdorff distance betweem clouds}

We present a lemma about the distance between clouds with intersecting stabilizers.

\begin{lemma}\label{lemmaDist}
If two clouds have a nontrivial intersection of their stabilizers, then the Gromov--Hausdorff distance between them can only be $0$ or $\infty$.
\end{lemma}

\begin{proof}
For any clouds $[X], [Y]$ and any $\lambda \in \mathbb{R}^{+}$, we have
$$
\big|\lambda[X], \lambda[Y]\big| = \lambda\big|[X], [Y]\big|.
$$
If $\lambda \neq 1$ belongs to the stabilizers of both clouds, then
$$
\big|[X],[Y]\big| = \big|\lambda[X], \lambda[Y]\big| = \lambda\big|[X], [Y]\big|.
$$
Since $\lambda \neq 1$, the quantity $\big|[X],[Y]\big|$ can only be $0$ or $\infty$.
\end{proof}

\begin{theorem}\label{thrmDist}
Let $[Z]$ be a cloud with a nontrivial stabilizer which has $Z$ as its center. Suppose there exist spaces $Y_{1}, Y_{2} \in [Z]$ such that
\begin{enumerate}
    \item $\max\{ |Y_{1},Z|, |Y_{2}, Z| \} = r > 0$,
    \item $|Y_{1}, Y_{2}| > r$.
\end{enumerate}
Then the Gromov--Hausdorff distance between clouds $[\Delta_1]$ and $[Z]$ is infinite.
\end{theorem}

\begin{proof}
The clouds $[\Delta_1]$ and $[Z]$ have stabilizers with nontrivial intersection. By Lemma \ref{lemmaDist}, the distance between them can only be $0$ or $\infty$.

To prove the theorem, it suffices to show that the Gromov--Hausdorff distance cannot be $0$. We need to establish that no correspondence with arbitrarily small distortion can exist between them. Let $R$ be a correspondence between $[\Delta_1]$ and $[Z]$ with $\operatorname{dis} R = \varepsilon < \infty$.

Fix $Y \in R(\Delta_1)$. By Theorem \ref{thrmCenterImage}, the Gromov--Hausdorff distance between $Y$ and $Z$ is at most $2\varepsilon$.

The theorem's conditions give:
$$
\max\{ |Y_{1},Z|, |Y_{2}, Z| \} = r < |Y_{1}, Y_{2}|.
$$
This implies there exists $c > 0$ such that $|Y_{1},Y_{2}| = (1 + c)r$. Consider the preimages $X_1 \in R^{-1}(Y_{1})$, $X_2 \in R^{-1}(Y_{2})$. We obtain:
$$
|X_1, \Delta_1| \leq |Y_{1}, Y| + \varepsilon \leq r + 3\varepsilon.
$$
The same is true for $X_2$, while:
$$
|X_1, X_2| \geq |Y_{1}, Y_{2}| - \varepsilon = (1+c)r - \varepsilon.
$$
By Remark \ref{remUltraMetric}:
$$
(1+c)r - \varepsilon \leq r + 3\varepsilon.
$$
which yields:
$$
\varepsilon \geq \frac{cr}{4}.
$$
This lower bound for $\varepsilon = \operatorname{dis} R$ shows the distortion cannot be arbitrarily small, hence the Gromov--Hausdorff distance cannot be $0$ and must therefore be infinite.
\end{proof}

\begin{corollary}
In the cloud $[\mathbb{R}]$, we can take $Y_{1} = \mathbb{Z}$ and $Y_{2} = \widetilde{\mathbb{R}}$. By Theorem \ref{thrmRUltraMetric}, they satisfy the conditions of Theorem \ref{thrmDist} with $r = \frac{1}{2}$. The stabilizer of $[\mathbb{R}]$ is $\mathbb{R}^{+}$ (nontrivial). Therefore $\big|[\Delta_{1}],[\mathbb{R}]\big| = \infty$.
\end{corollary}
\markright{References}

\end{document}